\documentclass[11pt]{amsart}

\usepackage{amssymb}
\usepackage{latexsym}
\usepackage{amsbsy}
\usepackage{amsfonts}
\usepackage{enumitem}

\newtheorem{theorem}[equation]{Theorem}
\newtheorem{proposition}[equation]{Proposition}
\newtheorem{lemma}[equation]{Lemma}
\newtheorem{corollary}[equation]{Corollary}
\newtheorem{definition}[equation]{Definition}

\theoremstyle{definition}
\newtheorem{auxremark}[equation]{Remark}
\newenvironment{remark}{%
  \begin{auxremark}%
  }{%
   \hfill$\diamondsuit$%
    \end{auxremark}
  }

\numberwithin{equation}{section}

\def\AArm{\fam0 \rm}%
\newdimen\AAdi%
\newbox\AAbo%
\def\AAk#1#2{\setbox\AAbo=\hbox{#2}\AAdi=\wd\AAbo\kern#1\AAdi{}}%

\newcommand{\BBone}{{\ensuremath{{\AArm 1\AAk{-.8}{I}I}}}}

\makeatletter
\def\eqlabel#1{\def\@currentlabel{#1}}

\def\formula#1{\def\@tempa{#1}\let\@tempb\theequation\def\theequation{%
\hbox{#1}}\def\@currentlabel{(\theequation)}$$}
\def\endformula{\leqno\hbox{(\@tempa)}$$\@ignoretrue\let\theequation\@tempb}

\def\given{\hskip5\p@\relax\vrule\@width.4\p@\hskip5\p@\relax}

\newcommand{\open}[1]{%
\par\normalfont\topsep6\p@\@plus6\p@\trivlist\item[\hskip\labelsep\itshape#1%
\@addpunct{.}]\ignorespaces}

\DeclareRobustCommand{\close}[1]{%
  \ifmmode 
  \else \leavevmode\unskip\penalty9999 \hbox{}\nobreak\hfill
  \fi
  \quad\hbox{$#1$}}
\makeatother

\newlength{\toskip}

\def\<{\langle}
\def\>{\rangle}

\def \R {{\mathbb R}}
\def \Q {{\mathbb Q}}

\def \P {{\mathbb P}}
\def \E {{\mathbb E}}
\def \N {{\mathbb N}}

\def \L {{\mathbb L}}

\def \Var {\textrm{Var}}

\def \Osc {\textrm{Osc}}

\newcommand*{\nrm}[1]{\left\|#1\right\|}
\newcommand*{\abs}[1]{\left|#1\right|}
\newcommand*{\vt}[1]{\nrm{#1}_{\text{\scshape\tiny TV}}}
\newcommand*{\dmu}{\,d\mu}
\DeclareMathOperator{\dvg}{div}

\begin{document}
\date{\today}

\title[Poincar\'e and hitting times.]{Poincar\'e inequalities and hitting times}

 \author[P. Cattiaux]{\textbf{\quad {Patrick} Cattiaux $^{\spadesuit}$ \, \, }}
\address{{\bf {Patrick} CATTIAUX},\\ Institut de Math\'ematiques de Toulouse. CNRS UMR 5219. \\
Universit\'e Paul Sabatier,
\\ 118 route
de Narbonne, F-31062 Toulouse cedex 09.} \email{cattiaux@math.univ-toulouse.fr}

\author[A. Guillin]{\textbf{\quad {Arnaud} Guillin $^{\diamondsuit}$}}
\address{{\bf {Arnaud} GUILLIN},\\ Institut Universitaire de France et Laboratoire de Math\'ematiques, CNRS UMR 6620, Universit\'e Blaise Pascal,
avenue des Landais, F-63177 Aubi\`ere.} \email{guillin@math.univ-bpclermont.fr}

\author[P. A. Zitt]{\textbf{\quad {Pierre Andr\'e} Zitt $^{\clubsuit}$}}
\address{{\bf {Pierre Andr\'e} ZITT},\\ Institut  de Math\'ematiques de Bourgogne, CNRS UMR 5584, Universit\'e de
Bourgogne, 9 avenue Alain Savary, B.P. 47870, F-21078 Dijon cedex.}
\email{Pierre-Andre.Zitt@u-bourgogne.fr}

\maketitle
 \begin{center}

 \textsc{$^{\spadesuit}$  Universit\'e de Toulouse}
\smallskip

\textsc{$^{\diamondsuit}$ Universit\'e Blaise Pascal and Institut Universitaire de France}
\smallskip

\textsc{$^{\clubsuit}$ Universit\'e de Bourgogne}
\smallskip
 \end{center}

\begin{abstract}
Equivalence of the spectral gap, exponential integrability of hitting times and Lyapunov
conditions are well known. We give here the correspondance (with quantitative results) for
reversible diffusion processes. As a consequence, we generalize results of Bobkov in the one
dimensional case on the value of the Poincar\'e constant for logconcave measures to
superlinear potentials. Finally, we study various functional inequalities under different
hitting times integrability conditions (polynomial, ...). In particular, in the one dimensional case,
ultracontractivity is equivalent to a bounded Lyapunov condition.
\end{abstract}
\bigskip

\textit{ Key words :}  Poincar\'e inequalities,  Lyapunov functions, hitting times,
log-concave measures, Poincar\'e-Sobolev inequalities.
\bigskip

\textit{ MSC 2010 : .} 26D10, 39B62, 47D07, 60G10, 60J60.
\bigskip

\maketitle

\section{Introduction.\protect\footnote{This work has benefited from the support of the Agence Nationale de la Recherche project EVOL.}}\label{Intro}

During the recent years a lot of progress 
has been made in the understanding of functional
inequalities and their links with the long time behavior of stochastic processes. Very recently,
starting with \cite{BCG}, the interplay between functional inequalities and the Lyapunov functions
used in the ``Meyn-Tweedie'' theory (\cite{DMT,MT}) has emerged (see \cite{BBCG,CGWW,CGGR} and the
recent survey \cite{CGgre}).

In the present paper we shall go a step further by showing the equivalence between the (usual)
Poincar\'e inequality, the existence of a Lyapunov function and the exponential integrability of
the hitting times of open bounded subsets.

As we shall recall below, this equivalence is well known in
the Markov chains setting, a key tool being the renewal theory. We shall discuss here the
diffusion process setting. In order to avoid technical intricacies,  we only look at
``very regular'' cases, i.e. hypoelliptic processes.

Note that the question of the existence of exponential moments for hitting times when a
Poincar\'e inequality holds was addressed in \cite{CarKlein} almost thirty years ago. We will
precise explicit values for the constants, and add Lyapunov functions to the picture.

The one
dimensional situation was recently discussed in \cite{LLS}, but as it is well known, monotonicity
arguments make things easier in the one dimensional situation.
\medskip

The main theorem is derived in Section \ref{secexit}. The proof being constructive, it allows us
to give quantitative estimates for hitting times as well as versions of the Poincar\'e inequality
where the mean is replaced by any ``local mean'' control. This is done in Section \ref{plus}. In
Section \ref{line} we look at the one dimensional setting. We show that Boltzmann-Gibbs measures
with a super-linear potential at infinity satisfy a Poincar\'e inequality and recover (up to the
universal constant) the control of the Poincar\'e constant for log-concave Probability measures
obtained by Bobkov (\cite{bob99}). In the final section \ref{secweak} we shall also discuss
polynomial moments of hitting times, instead of exponential ones, in connection with weak
Poincar\'e inequalities. This section is reminiscent of the work of Mathieu (\cite{Math}).

\section{Poincar\'e inequality and hitting times.}\label{secexit}

\subsection{The main result}
Let us first recall the known situation for Markov chains. For simplicity assume that the state
space $E$ is countable, and that $Q$ is a Markov transition kernel on $E$ which is irreducible
and aperiodic. Denote by $(X_n)_{n \in \N}$ the associated Markov chain. For $a\in E$ we denote by
$T_a$ the hitting time of $\{a\}$ i.e. $T_a=\inf\{n\geq 0 \, ; \, X_n=a\}$.
Then

\begin{theorem}\label{thmmainchain}
Under the previous assumptions, the following statements are equivalent
\begin{enumerate}
\item there exist $a\in E$ and  $\rho>1$ such that for all $x\in E$,
$\E_x\left(\rho^{T_a}\right)<+\infty \, ,$
\item there exist an invariant probability measure
$\pi$ and $0<\theta<1$ such that for all $x\in E$ one can find $C(x)$ with $$
\vt{ Q^n(x,.)  -  \pi(.)}   \leq  C(x) \, \theta^n \, ,$$ where $\vt{ \nu - \mu}$
denotes the total variation distance between $\mu$ and $\nu$,
\item there exists a Lyapunov
function,  i.e. a function $W:E\to\R$, such that $W\geq 1$,  $(Q-Id)W  :=  LW  \leq  \alpha  W + b
\BBone_a$ for some $0<\alpha<1$ and some $b\geq 0$.
\end{enumerate}
In addition if the (unique) invariant measure is symmetric, these statements are equivalent to the
following two additional ones
\begin{enumerate}[resume] 
\item there exists a constant $C_P$ such that the Poincar\'e inequality $$\Var_\pi(f) \, \leq
\, C_P \, \langle (Id - Q^2)f \, , \, f\rangle$$ holds for all $f\in l^2(\pi)$
($\langle.,.\rangle$ being the scalar product in $l^2(\pi)$),
\item there exists some
$0<\lambda<1$ such that $\Var_\pi(Q^nf) \, \leq \, \Var_\pi(f) \, \lambda^{2n}$.
\end{enumerate}
\end{theorem}

The equivalence between (1) and (3) is an exercise, while (3) implies (2) can be nicely shown as
remarked by M. Hairer and J.C. Mattingly (\cite{HMyet}) even in a stronger form. The converse (2)
implies (1) is more intricate, and usual proofs call upon Kendall's renewal theorem and an argument
of analytic continuation (see e.g. S.~Meyn and R.~Tweedie's monograph \cite{MT}).
In particular we can give explicit expressions for the constants
for all implications, except this one (i.e. if $(2)$ holds, we only know that $(1)$ holds for some non explicit $\rho$.)

The equivalence between (4) and (5) is well known, while (5) clearly implies (2). Finally, (3)
implies that (2) holds for $Q$ hence for $Q^2$  changing $\theta$. Hence (3) holds for $Q^2$, and
this implies that the Poincar\'e inequality (4) holds according to an argument due to Mu-Fa Chen
(\cite{chenbook} p. 221-235).
\medskip

The aim of this section is to extend this result to some continuous time diffusion processes on
$\R^d$ (or a finite dimensional Riemannian manifold). We also want to get bounds for all the
constants, as precisely as
possible. Actually, an accurate study of the literature provides (in
possibly more general situations) almost all the results we shall state. One possible way is to
use some skeleton chain and Theorem \ref{thmmainchain} (with some loss for the constants). Our
approach will be more direct and elementary.
\medskip

For simplicity we shall consider $\R^n$ valued diffusion processes $(X_t)_{t>0}$ with generator
$$L = \sum_{i,j} \, a_{ij} \, \partial^2_{ij} \, + \, \sum_i \, b \, \partial_i$$ where
$a=\sigma^* \, \sigma$, $\sigma_{ij}$ and $b_i$ being smooth enough ($C^\infty$ for instance). We
introduce the ``carr\'e du champ'' operator
$$\Gamma(f,g) = \frac 12 \, \left(L(fg) - f Lg - gLf \right) = \langle \sigma \, \nabla f \, , \, \sigma \, \nabla g\rangle \, .$$
In addition we assume that $\mu(dx)=e^{-V(x)} dx$ is a symmetric probability measure for the
process, where the potential $V$ is also assumed to be smooth. Thus $L$ generates a
$\mu$-symmetric semi-group $P_t$ and the $\L^2$ ergodic theorem (in the symmetric case) tells us
that for all $f\in \L^2(\mu)$, $$\lim_{t \to +\infty} \, \parallel P_tf \, - \, \int f \,
\dmu\parallel_{\L^2(\mu)} \, = \, 0 \, .$$ If $U$ is an open subset of $\R^d$ we define $$T_U
=\inf\{t>0 \, ; \, X_t \in U\} \, .$$

Consider the following statements:
\begin{enumerate}
\item[(H1)] There exists a Lyapunov function $W$, i.e. there exist a smooth function $W:\R^n \to \R$, s.t. $W\geq 1$, a
constant $\lambda > 0$ and an open connected bounded subset $U$ such that
\[ LW \, \leq \, - \,
\lambda \, W \quad \textrm{ on } \, \, (\bar{U})^c \, .
\]
\item[(H2)] There exist an open
connected bounded subset $U$ and a constant $\theta>0$ such that for all $x$,
\[
\E_x\left(e^{\theta \, T_U}\right) < + \infty \, ,
\]
and $x \mapsto \E_x\left(e^{\theta \, T_U}\right)$ is locally bounded.
\item[(H2$\mu$)] There exist an open connected bounded subset $U$
and a constant $\theta>0$ such that,
$$\E_\mu\left(e^{\theta \, T_U}\right) < + \infty \, .$$
\item[(H3)] There exist constants $\beta >0$ and $C>0$ and a
function $W\geq 1$ belonging to $\L^1(\mu)$ such that for all $x$
$$\vt{ P_t(x,.) - \mu } \, \leq \, C \, W(x) \, e^{- \beta \, t} \, .$$
\item[(H4)] $\mu$ satisfies a Poincar\'e inequality, i.e. there exists a constant $C_P$ such that
for all smooth $f$,
$$\Var_\mu(f) \, \leq \, C_P \, \int \, \Gamma(f,f) \dmu \, .$$
\item[(H5)] There exist
constants $\eta >0$ and $C>0$ such that for all bounded $f$, $$\Var_\mu (P_tf) \, \leq \, C \,
e^{-\eta \, t} \,  \Osc^2(f) \, ,$$ where $\Osc(f)$ denotes the oscillation of $f$.
\item[(H6)]
There exists a constant $C_S$ such that for all $f\in \L^2(\mu)$,
\[
\Var_\mu(P_tf) \, \leq \, e^{-
\, C_S \, t} \, \Var_\mu(f) \, .
\]
\end{enumerate}

Finally we also introduce the following definition
\begin{definition}\label{defhypostrong}
We shall say that $L$ is strongly hypoelliptic if it can be written in H\"{o}rmander form
$L=\sum_j \, X_j^2 + Y$ where the $X_j$'s and $Y$ are smooth vector fields such that the Lie
algebra generated by the $X_j$'s is full at each $x\in \R^n$ (i.e. spans the tangent space at each
$x$). Note that in this situation $\Gamma(f,f) = \sum_j |X_j f|^2$.

We shall say that $L$ is uniformly strongly hypoelliptic if all the $X_j$'s are bounded with
bounded derivatives (of any order) and there exist $N \in \N$, $\alpha
>0$ such that for all $\xi \in \R^n$, $$\sum_{Z\in L_N(x)} \, \langle Z(x),\xi\rangle^2 \geq
\alpha |\xi|^2$$ where $L_N(x)$ denotes the set of Lie brackets of length smaller or equal to $N$
computed at $x$.
\end{definition}

We may state now our main

\begin{theorem}\label{thmmaindiff}
The following relations hold true (recall that $\mu$ is symmetric)
\begin{enumerate}
\item (H1) $\Rightarrow$ (H3) $\Rightarrow$ (H4) $\Leftrightarrow$ (H5) $\Leftrightarrow$ (H6),
\item (H1) $\Rightarrow$ (H2) and (H2 $\mu$).
\item If $L$ is uniformly strongly hypoelliptic then (H4) $\Rightarrow$ (H2) and (H2 $\mu$), and (H2) or (H2 $\mu$) $\Rightarrow$
(H1).
\end{enumerate}
Hence if $L$ is uniformly strongly hypoelliptic all statements (H1) up to (H6) are equivalent.
\end{theorem}


Let us make a few remarks on the hypotheses.

\begin{remark}[Hypo-ellipticity] \quad

  In particular, the diffusion with a gradient drift $L = \Delta - \nabla V \cdot \nabla$
  is of course hypo-elliptic. We will see later precise computations for the constants in this case, under
  the additional assumption:
  \begin{equation}
     LV + \frac{1}{2} \Gamma(V,V) \leq C_m < \infty.
    \label{eq=hypotheseCm}
  \end{equation}
\end{remark}

\begin{remark}[Symmetry]\label{remnonsym}
Actually several implications are still true without the symmetry assumption. However symmetry is
required for (H5) $\Rightarrow$ (H6) (counter-examples are known in the non-symmetric situation,
see e.g. \cite{BCG} section 6 with kinetic Fokker-Planck equations). It is also required for our proof of (H1) $\Rightarrow$ (H4), but
it is not for the one in \cite{DMT,DFG}. \\
Symmetry is used in the proof of (H4)
$\Rightarrow$ (H2), but it is not required for the first partial result i.e. the existence of the
exponential moment for $\mu$ almost all $x$ (which holds in much more general cases according to
the framework of \cite{CatGui2}). This result appears in the paper by Carmona and Klein
(\cite{CarKlein}) where the exponential integrability of hitting times is shown under exponential
rate of convergence in the ergodic theorem (hence Poincar\'e) and we are able to give a precise
bound for the exponent (answering the question in Remark 2 of \cite{CarKlein}).\\
Note also that the implications (H1) to (H5) holds also, with additional assumptions
(local Poincar\'e inequality and (slight) conditions on the constants involved in (H1)) using Lyapunov-Poincar\'e inequalities as in \cite{BCG}.\\
Let us finally remark that Rockner-Wang \cite{rw} proves (H5) to (H6) without symmetry but assuming that $L$ is normal (i.e. $LL^*=L^*L$).
\end{remark}

\begin{remark}\label{remcompare}
Of course, provided $W$ is everywhere defined and smooth, (H1) can be rewritten: there exists a
Lyapunov function $W$, i.e. there exist a smooth function $W\geq 1$, a constant $\lambda > 0$ and
an open connected bounded subset $U$ such that $$LW \, \leq \, - \, \lambda \, W  + b \,
\BBone_{U} \, ,$$ with $b = \sup_U \left(LW + \lambda W\right)$. This formulation is the one used
in \cite{BBCG} yielding another bound for the Poincar\'e constant, namely
\begin{equation}\label{eqpoincbbcg}
C_P \leq \frac 1 \lambda \, (1 + b C_P(U)) \, .
\end{equation}
The bound we will get below (eq. \eqref{eqpoinclyap}) is not immediately comparable with this one.

In particular if
(H2) holds in our strong hypoelliptic framework, $x \mapsto \E_x(e^{\theta T_U})$ is smooth
(provided the boundary $\partial U$ is non characteristic) on $\overline{U^c}$ (see again
\cite{cathypo2}) hence can be smoothly extended to the whole $R^n$ according to Seeley's theorem.
But an explicit bound for $b$ is difficult to obtain.
\end{remark}

\subsection{Proof of the main theorem}
Let us begin by a small remark on (H1).
\begin{remark}[Integrability of $W$]\label{reminteg}
We did not impose any integrability condition for $W$ in (H1). Actually if $W$ satisfies (H1), $W$
automatically belongs to $\L^1(\mu)$.

Indeed choose some smooth, non-decreasing, concave function $\psi$ defined on $\R^+$, satisfying
$\psi(u)=u$ if $u\leq R$, $\psi(u)= R+1$ if $u\geq R+2$ and with $\psi'(u) \leq 1$ (such a
function exists). Then $\psi(W)$ is smooth and bounded. According to the chain rule
\begin{equation}
  \label{eqchainrule}
L(\psi(W)) = \psi'(W) \, LW + \psi''(W) \, \Gamma(W,W) \, \leq \, - \lambda \, \psi'(W) \, W \quad \textrm{ on
} \bar{U}^c \, ,
\end{equation}
thanks to our assumptions. For $R$ large enough, $W \leq R$ on $U$, so that
$\psi(W)=W$ on $U$. It follows

\begin{align*}
\lambda \, \int  W \, \BBone_{W\leq R} \dmu
& \leq  \lambda \, \int  \psi'(W) \, W \dmu\\
&  =   \int  L(\psi(W)) \dmu + \lambda \, \int \, \psi'(W) \, W \dmu & &\text{since ($\int Lg \dmu = 0$)}\\
& \leq \int_U  \left(L(\psi(W))  + \lambda \,  \psi'(W) \, W \right) \dmu & &\text{using \eqref{eqchainrule}} \\
& \leq \int_U  \left(LW + \lambda \, W\right) \dmu \, = \, C(U) \, ,
\end{align*}
where $C(U)$ does not depend on $R$. We conclude by letting $R$ go to $\infty$.
\end{remark}

We now turn to the proof of the theorem.

\textit{(H4) $\Leftrightarrow$ (H6).} This is well known and we have in addition $C_S=2/C_P$.
\smallskip

\textit{(H6) $\Leftrightarrow$ (H5).} (H6) clearly implies (H5). Since $\mu$ is symmetric the
converse is proven in \cite{rw} using the spectral resolution. For the sake of completeness we
shall give below a very elementary proof of this fact based on the following
\begin{lemma}\label{lemlogsemi}
$t \mapsto \log \parallel P_t f\parallel_{\L^2(\mu)}$ is convex.
\end{lemma}
Indeed if $n(t)=\parallel P_t f\parallel^2_{\L^2(\mu)}$, the sign of the second derivative of
$\log n$ is the one of $n'' n - (n')^2$. But $$n'(t)= 2 \, \int \, P_tf \, LP_t f \dmu$$ and
$$n''(t) = 2 \, \int \, (LP_tf)^2 \dmu + 2 \, \int \, P_tf \, LP_t Lf \dmu = 4 \, \int \,
(LP_tf)^2 \dmu \, ,$$ so that the lemma is just a consequence of Cauchy-Schwarz inequality.

This convexity is a key argument in the proof of the following
\begin{lemma}\label{lemdecayf}
  Let $\mathcal{C}$ be a dense subset of $\L^2(\mu)$. Suppose that there exists $\beta>0$, and, for any $f \in \mathcal{C}$, a constant
  $c_f$ such that:
  \[\forall t, \quad  \Var_\mu(P_t f) \leq c_f e^{-\beta t}.\]
  Then
  \[ \forall f \in \L^2(\mu), \forall t, \quad \Var_\mu (P_t f) \leq e^{-\beta t}\Var_\mu(f).\]
\end{lemma}

Our claim (H5) implies (H6) immediately follows with $\eta = C_S$. In order to prove lemma
\ref{lemdecayf}, assuming that $\int f \dmu=0$ which is not a restriction, it is enough to look
at
$$t \mapsto \log
\parallel P_t f\parallel_{\L^2(\mu)} + (\beta t/2) \, ,$$ which is convex, according to lemma
\eqref{lemlogsemi}, and bounded since $\Var_\mu (P_t f) \leq c_f \, e^{-\beta t}$. But a bounded
convex function on $\R^+$ is necessarily non-increasing. Hence $$\parallel P_t
f\parallel_{\L^2(\mu)} \leq e^{-\beta t/2} \, \parallel P_0 f\parallel_{\L^2(\mu)}$$ for all $f\in
\mathcal C$, the result follows using the density of $\mathcal C$.
\medskip

\textit{(H3) $\Rightarrow$ (H5).} This is shown in \cite{BCG} Theorem 2.1 and we may choose the
constant $C$ in (H5) equal to $8 C \int W \dmu$ where $C$ is the constant in (H3), and
$\eta=\beta$.
\smallskip

\textit{(H1) $\Rightarrow$ (H3).} This is the key result in \cite{DMT} (also see \cite{DFG}),
  unfortunately with an essentially non explicit control of the constants.

Combining all these results we get the first statement of the theorem, in particular we already
know that (H1) implies (H4).

A direct and short proof of (H1) $\Rightarrow$ (H4) is given in \cite{BBCG} for $L=\Delta - \nabla
V.\nabla$ which is the natural symmetric operator associated with $\mu$. Let us give a
slightly modified proof, yielding a better control on the constants and extending it to more
general operators.

The key is the following ($f$ being smooth)
\begin{equation}\label{eqipp}
\int \, \frac{-LW}{W} \, f^2 \dmu \, \leq \, \int \, \Gamma(f,f) \dmu \, ,
\end{equation}
which is a consequence of
\begin{align*}
\int \, \frac{-LW}{W} \, f^2 \dmu & =  \int \, \Gamma\left(\frac{f^2}{W},W\right) \dmu \\
& =  2 \, \int \, \frac fW \, \Gamma(f,W) \dmu \, - \, \int \, \frac{f^2}{W^2} \, \Gamma(W,W) \dmu \\
& =  - \, \int \abs{ \frac fW \, \sigma \nabla W - \sigma \nabla f}^2 \dmu \, + \, \int \,
\Gamma(f,f) \dmu \, .
\end{align*}
Next for $r>0$ introduce $U_r=\{x \, ; \, d(x,U)<r \}$ for the (euclidean or riemannian) distance
$d$. Let $0 \leq \chi \leq 1$ be a $C^\infty$ function such that $\chi=1$ on $U$ and $\chi=0$ on
$U_r^c$. Then
\begin{align*}
\int \, f^2 \dmu & =  \int ( f(1-\chi)+ f \chi)^2 \dmu \\
& \leq  2  \int  f^2
(1-\chi)^2 \dmu + 2 \int f^2  \chi^2 \dmu \\
 & \leq  \frac 2\lambda \, \int \frac{-LW}{W} \,
f^2 \, (1-\chi)^2 \dmu + 2 \, \int_{U_r} \, f^2 \dmu \\
 & \leq  \frac 2\lambda \, \int
\Gamma(f(1-\chi),f(1-\chi)) \dmu + 2 \, \int_{U_r} \, f^2 \dmu
\end{align*}
by \eqref{eqipp}. Since $\Gamma(fg,fg) \leq 2(f^2 \,\Gamma(g,g) + g^2 \,\Gamma(f,f))$, we get:
\begin{align*}
\int \, f^2 \dmu
 & \leq  \frac 4\lambda \,
\int \Gamma(f,f) \dmu + \frac 4\lambda \, \int f^2 \, \Gamma(\chi,\chi) \dmu + 2 \,
\int_{U_r} \, f^2 \dmu \\
 & \leq  \frac 4\lambda \, \int \Gamma(f,f) \dmu + \left(\frac {4
 \nrm{\Gamma(\chi,\chi)}_\infty}{\lambda} + 2\right) \, \int_{U_r} \, f^2 \dmu \, .
\end{align*}
Now, if $\mu$ satisfies a Poincar\'e inequality in restriction to $U_r$, i.e. $$\int_{U_r} \, f^2
\dmu \, \leq \, C_P(U,r) \, \int_{U_r} \Gamma(f,f) \dmu \quad \textrm{ if } \quad \int_{U_r}
\, f \dmu = 0 \, ,$$ we may apply the previous inequality with $g=f -\int_{U_r} f \dmu$,
yielding, since $\sigma \nabla f =\sigma \nabla g$,
\begin{equation}\label{eqpoinclyap}
\Var_\mu(f) \, \leq \, \int g^2 \dmu \, \leq \, \left(\frac 4\lambda + \left(\frac {4 \nrm{
\Gamma(\chi,\chi)}_\infty}{\lambda} + 2\right) C_P(U,r)\right) \, \int \, \Gamma(f,f) \dmu \, ,
\end{equation}
i.e. the Poincar\'e inequality (H4). Note that we may always replace $U$ by a larger euclidean
ball, i.e. we may assume that $U$ is an euclidean ball. According to the discussion in \cite{BCG}
p.744-745, if $L$ is strongly hypoelliptic, $\mu$ satisfies the Poincar\'e inequality in
restriction to any euclidean ball, so that we have shown that (H1) $\Rightarrow$ (H4) in this
case.
\medskip

We now turn to the part of the results involving the stochastic process.

\textit{(H1) $\Rightarrow$ (H2).} This is a simple application of Ito's formula applied to $(t,x)
\mapsto e^{at} \, W(x)$ (notice that (H1) implies that the diffusion process is non-explosive or
conservative). Indeed let $x\in U^c$, and $a \leq \lambda$. Define $T_{UR}$ as the first hitting
time of $U \cup \{|y|>R\}$. For $R>|x|$ we thus have
\begin{eqnarray*}
\E_x\left(e^{a(t\wedge T_{UR})}\right) & \leq & \E_x\left(e^{a(t\wedge T_{UR})} \, W(X_{t\wedge
T_{UR}})\right) \\ & \leq & W(x) + \E_x\left(\int_0^{t\wedge T_{UR}} \, (aW + LW)(X_s) \, e^{as}
\, ds \right)\\ & \leq & W(x) + \E_x\left(\int_0^{t\wedge T_{UR}} \, (a - \lambda) W(X_s) \,
e^{as} \, ds \right)\\ & \leq & W(x) \, ,
\end{eqnarray*}
so that letting first $R$ then $t$ go to infinity we obtain (H2) for $\theta = \lambda$, thanks to
Lebesgue's monotone convergence theorem.

The same proof shows that (H2 $\mu$) holds since we know that $W \in \L^1(\mu)$.
\smallskip

Conversely, assume (H2) and the strong hypoellipticity of $L$. Again we may assume that $U$ is
an euclidean ball so that for any $R>0$, the boundary of the euclidean shell $U_R - U$ is
non-characteristic for $L$. We may then use the results in e.g. \cite{cathypo2} Theorem 5.14
(local boundedness in (H2) ensures that hypothesis (HC) in \cite{cathypo2} is satisfied),
showing that
$$x \mapsto W_R(x) = \E_x\left(e^{\theta (T_U\wedge T_{U_R^c})}\right)$$ is smooth and solves the
Dirichlet problem $$LW_R + \theta \, W_R = 0 \quad \textrm{ in } \, U_R - U \quad , \quad
W_R=1 \quad \textrm{ on } \partial(U_R -U) \, .$$ Using (H2) again it then follows that $$x
\mapsto W(x) = \E_x\left(e^{\theta \, T_U}\right)$$ is well defined, solves the Dirichlet
problem with $R=+\infty$ in the sense of Schwartz distributions, hence is smooth thanks to
hypoellipticity. $W$ is then a Lyapunov function in (H1). If (H2 $\mu$) is satisfied, then an
 argument below will show that (H2) is satisfied.
\medskip

To conclude the proof of the theorem it remains to show that the Poincar\'e inequality (H4)
implies (H2). Let $U$ be an open bounded set. The idea is that, if $T_U$ is large, the process
stays for a long time in $U^c$, and spends no time at all in $U$. However, the ergodic
properties given by the Poincar\'e inequality tell us that, for large times, the fraction of
the time spent in $U$ should be proportional to $\mu(U)$; therefore $T_U$ cannot be too large.

To be more precise,
\begin{equation}\label{eqretour a ergod}
\left\{T_U > t \right\} \, \subseteq \left\{ \frac 1t \, \int_0^t \, \BBone_U(X_s) \, ds = 0\right\}  .
\end{equation}
Hence
\begin{align}
  \notag
\P_\nu(T_U>t) & \leq  \P_\nu \left(\frac 1t \, \int_0^t \, \BBone_U(X_s) \, ds = 0\right) \\
\notag
& \leq  \P_\nu \left(%
    -  \frac 1t \, \int_0^t \, \BBone_U(X_s) \, ds + \mu(U) \geq \mu(U)
  \right)\\
\label{eqQueueDeT}
& \leq  \nrm{ \frac{d\nu}{d\mu}}_{\L^2(\mu)} \cdot \exp\left( - \frac{t \,\mu(U)}{8C_P \, (1-\mu(U))} \right) ,
\end{align}
provided $\mu(U)\leq 1/2$. The latter is a consequence of Proposition 1.4 and Remark 1.6 in
\cite{CatGui2}.

From there, we get exponential moments, using the elementary lemma:
\begin{lemma}
  For any positive random variable,
  \[ \E[e^{\theta T}] = 1 + \int_0^\infty \theta e^{\theta t} \P[T>t] dt.\]
  If for some $s_0, \theta_U$ and for $t>s_0$, $\P[T>t] \leq C\exp( - (t-s_0)\theta_U)$,
  then
  \[\forall \theta < \theta_U, \quad
  \E[e^{\theta T} ] \leq e^{\theta s_0} \left( 1 + C \frac{\theta}{\theta_U - \theta} \right).\]
  \label{lmm=lemmeTrivial}
\end{lemma}
For $s_0 = 0$, $\theta_U = \mu(U)/8 C_P$, and $\nu = \mu$,  using \eqref{eqQueueDeT} and this lemma,
we get $\E_\mu(e^{\theta T_U})<+\infty$, for any $\theta < \theta_U$. This entails that
$\E_x[e^{\theta T_U}]$ is itself finite, for $\mu$-almost any $x$.

If we assume the uniform strong hypoellipticity the marginal law at time $t$ of $\P_x$ has an
everywhere positive smooth density $r(t,x,.)$ w.r.t.  $\mu$, and symmetry combined with the
Chapman-Kolmogorov relation yield
\[  \int r^2(t,x,y) \, \mu(dy) =  r(2t,x,x) <\infty,\]
showing
that the $\P_x$ law of $X_1$ has a density $r(1,x,.)\in \L^2(\mu)$. We may thus apply the previous
result with $\nu= r(1,x,.) \mu$.

Notice that this argument also shows that $T_U$ has an exponential moment of order $\theta/2$ for
$\P_x$ as soon as it has an exponential moment of order $\theta$ for $\P_\mu$, i.e. (H2 $\mu$)
implies (H2).

\begin{remark}\label{remexplicitbound}
  The proof shows that $H4$ implies $H2$, i.e. the hitting times have finite exponential moments,
  but do not give explicit bounds on the value of these moments (depending on $x$). Such explicit
  bounds will be given in the next section.
\end{remark}

\section{Some consequences.}\label{plus}

We rephrase here the implication $H4 \implies H2$ of the main theorem, and add explicit
computations of the constants, and the dependence on $x$ of the moments,  in special cases.
\begin{proposition}\label{propentry}
Assume that the Poincar\'e inequality holds with constant $C_P$.

Then for all open set $U$ with $\mu(U) \leq 1/2$, $\E_x(e^{\theta T_U}) < +\infty$ for $$\theta <
\mu(U)/8 C_P \, (1 - \mu(U)) := \theta(U) \, .$$ If $\mu(U) \geq 1/2$ we may take
$\theta(U)=\mu^2(U)/2 \, C_P$.

If the boundedness assumption \eqref{eq=hypotheseCm} holds, there exists $C$ such that:
\begin{equation}
  \label{eq=boundOnTheMoment}
  \forall x,\forall \theta<\theta_U, \quad \E_x [e^{\theta T_U}] \leq C\left( 1 + e^{V(x)/2} \frac{\theta}{\theta_U - \theta}\right).
\end{equation}
If, in addition, we are in the elliptic case $L = \Delta - \nabla V \cdot \nabla$, \eqref{eq=boundOnTheMoment} holds
with $C$ replaced by $e^{\theta s_0}$, where $s_0 = \frac{1}{2\pi} e^{2C_m/n}$.
\end{proposition}
\begin{proof}

The first statement has already been proved.

If we assume the additional boundedness hypothesis (eq. \eqref{eq=hypotheseCm}),
  we can use stochastic calculus to get good bounds: the idea is that the density of the law
  of $X_t$ with respect to $\mu$ is computable, and its $L^2$ norm can be bounded.

 First of all recall that $L=\sum X_j^2 + Y$. Since
$\mu$ is symmetric
$$Y = \sum_j \dvg X_j \,  X_j - \sum_j X_j V \, X_j \, .$$ If we denote by $\Q_x$ the law of the
diffusion process starting from $x$ with generator $$L' = \sum_j X_j^2 + \sum_j \dvg X_j \, X_j$$ we have a
Girsanov type representation
\begin{eqnarray*}
G_t : = \frac{d\P_x}{d\Q_x}|_{\mathcal F_t} & = & \exp \left(- \frac 12 \, \int_0^t \, \langle X_j
V(\omega_s),d\omega_s\rangle \, - \, \frac 14 \, \int_0^t \, \Gamma(V,V)(\omega_s) \, ds\right) \\
& = & \exp \left(\frac 12 \, V(\omega_0) - \frac 12 V(\omega_t) - \, \frac 12 \, \int_0^t \,
(\frac 12 \, \Gamma(V,V)- L'V)(\omega_s) \, ds\right) \, ,
\end{eqnarray*}
the latter (Feynman-Kac representation) being obtained by integrating by parts the stochastic
integral. We can now follow an argument we already used in previous works. We write the details for
the sake of completeness.

Thanks to the uniform strong hypoellypticity we know that the marginal law at time $t$ of $\Q_x$
has an everywhere positive smooth density $q(t,x,.)$ w.r.t. Lebesgue measure satisfying for some
$M$ (see e.g. \cite{cathypo1} theorem 1.5) $$|q(t,x,y)|\leq C \, (1\wedge t)^{-M} \quad \textrm{
for all } \quad x,y \in \R^n \, .$$ Hence
\begin{align*}
\E_x[f(X_t)] & =  \E^{\Q_x}[f(\omega_t) \E^{\Q_x}[G|\omega_t]] \\
& = \int \, f(y) \, \E^{\Q_x}[G|\omega_t=y] \, q(t,x,y) \, dy \\
& = \int \, f(y) \, \E^{\Q_x}[G|\omega_t=y] \, q(t,x,y) \, e^{V(y)} \, \mu(dy) \, ,
\end{align*}
In other words, the law of $X_t$  has a density with respect to $\mu$ given by
\[ r(t,x,y)=\E^{\Q_x}[G|\omega_t=y] \, q(t,x,y) \, e^{V(y)} \, .\]
Hence
\begin{align*}
\int_0^{+\infty} \, r^2(t,x,y) \, \mu(dy)
& = \int \, \left(\E^{\Q_x}[G|\omega_t=y] \, q(t,x,y) \, e^{V(y)}\right)^2 \, e^{-V(y)}  \, dy \\
& = \E^{\Q_x} \left[q(t,x,\omega_t) \, e^{V(\omega_t)} \, \left(\E^{\Q_x}[G|\omega_t]\right)^2 \right] \\
& \leq \E^{\Q_x} \left[ q(t,x,\omega_t) \, e^{V(\omega_t)} \, \E^{\Q_x}[G^2|\omega_t] \right] \\
& \leq  e^{V(x)} \, \E^{\Q_x} \left[
    q(t,x,\omega_t) \, e^{- \int_0^t \, (\frac 12 \, \Gamma(V,V)- L'V)(\omega_s) ds}
    \right] \\
& \leq  C \, (1\wedge t)^{-M} \, e^{V(x)} \, e^{C_m t} \, .
\end{align*}
Hence the law at time 1 of $X_.$ has a density belonging to $\L^2(\mu)$. Using the result in
\cite{CatGui2} we have recalled and the Markov property we thus have for $t>1$
\[
\P_x(T_U>t) \leq
D \, e^{V(x)/2} \, e^{- \, \frac{(t-1) \, \mu(U)}{8C_P (1 - \mu(U))} }
\]
hence the result by lemme \ref{lmm=lemmeTrivial}.

Finally, if $L=\Delta - \nabla V.\nabla$,  we can be even more precise.

  Indeed $q(t,x,y)
  \leq (2\pi t)^{-n/2}$ so that for $t>s>0$, using \eqref{eqQueueDeT} we obtain
\[
\P_x(T_U>t) \leq (2\pi s)^{-n/4} \, e^{C_m s/2} \,  e^{V(x)/2} \, e^{- \, (t-s) \,
\theta(U) } \, .
\]
Choosing $s_0=\frac{1}{2\pi} \,e^{2C_m/n}$ we get for $t>s_0$,
\[ \P_x(T_U>t) \leq  e^{V(x)/2} \, e^{- \,(t-s_0) \, \theta(U) } \, ,
\]
so that for $\theta < \theta(U)$, using lemma \ref{lmm=lemmeTrivial}, we get
\[
\E_x\left(e^{\theta \, T_U}\right)\leq e^{\theta s_0} \, \left(1 + \frac{\theta}{\theta(U) -
\theta}e^{V(x)/2}\right).
\]

If $\mu$ satisfies a Poincar\'e inequality with constant $C_P$, so does $\mu^{\otimes k}$ for any
$k\in \N^*$. It thus follows as before that for $x=(x_1,...,x_k)$ and $\theta < \theta(U)$,
\[
  \P_x(T_U>t) \leq (2\pi s)^{-nk/4} \, e^{k \, C_m s/2} \,  e^{\sum_i V(x_i)/2} \, e^{- \, (t-s) \,
\theta(U) } \, ,
\]
so that for the same $s_0$,
\[
  \E_x\left(e^{\theta \, T_U}\right)\leq e^{\theta s_0} \, \left(1 + \frac{e^{\sum_i V(x_i)/2}}{\theta(U) -
\theta}\right).\]

\end{proof}

\begin{remark}\label{remstokes}
In the same way, when $L=\Delta - \nabla V.\nabla$, one can improve upon the constant if we assume
in addition that
\begin{equation}\label{eq:condstokes}
  \begin{array}{c}
\textrm{$U$ has a smooth boundary and } \frac{\partial W}{\partial n} \leq 0 \textrm{ on }
\partial U \\
 \textrm{ where $n$ denotes the outward normal to the boundary.} \end{array}
\end{equation}

Indeed in this case we can directly integrate by parts in $U^c$ using Stokes theorem. This yields
\begin{eqnarray*}
\int_{U^c} \, f^2 \dmu & \leq & \frac{1}{\lambda} \, \int_{U^c} \, f^2 \, \frac{-LW}{W} \,
e^{-V} \, dx \\ & \leq & \frac{1}{\lambda} \,  \int_{U^c} \, \frac{f^2}{W} \, (-\Delta W + \nabla
V.\nabla W) \, e^{-V} \, dx \\ & \leq & \frac{1}{\lambda} \,  \int_{U^c} \,
\left(\nabla\left(\frac{f^2 \, e^{-V}}{W}\right).\nabla W \, e^V + \frac{f^2}{W} \, \nabla
V.\nabla W\right) \, e^{-V} \, dx \, \\ & & + \, \frac{1}{\lambda} \, \int_{\partial U} \,
\left(\frac{f^2 \, e^{-V}}{W}\right) \, \frac{\partial W}{\partial n} \, dm_{\partial U}\\ & \leq
& \frac{1}{\lambda} \,  \int_{U^c} \, \nabla\left(\frac{f^2}{W}\right).\nabla W  \, e^{-V} \,
dx\\& \leq & \frac{1}{\lambda} \,  \int_{U^c} \, \left(|\nabla f|^2 -|\nabla f - \frac{f}{W}
\nabla W|^2\right) \dmu \\ & \leq & \frac 1\lambda \, \int_{U^c} \, |\nabla f|^2  \dmu \, .
\end{eqnarray*}
Therefore we  obtain in this case
\begin{equation}\label{eqpoincstokes}
C_P \leq \frac 1\lambda + C_P(U) \, ,
\end{equation}
which is of course much better than any other bound we gave.
\end{remark}
\bigskip

Recall that if $m_\mu(f)$ denotes a $\mu$ median of $f$, one has
\begin{equation}\label{eqmeanmed}
\Var_\mu(f) \leq \E_\mu[(f-m_\mu(f))^2] \leq 2 \, \Var_\mu(f) \, ,
\end{equation}
so that one may replace the variance by the squared distance to any median in Poincar\'e
inequality up to some universal constants. Using our previous results, we shall see that we may
replace the mean of $f$ by local means or values. Here is a first result in this direction.

\begin{theorem}\label{thmpoincdelocal}
Let $d\mu=e^{-V} \, dx$ be a probability measure satisfying a Poincar\'e inequality with constant
$C_P$, $a\in \R^n$ and $r>0$.

We assume that one can find a sequence $V_k$ of smooth functions such that $d\mu_k = e^{-V_k} dx$
 converges weakly to $\mu$ and $V_k$ converges uniformly  to $V$ on $B(a,2r)$.

Then there exists an universal constant $\kappa$ such that for all $f\in C_b^1$ with
$\int_{B(a,2r)} f \dmu = 0$, the following inequality holds $$\int \, f^2 \dmu \, \leq \,
\left(\frac{32 \, C_P}{\mu(B(a,r))} \, \left[1 + 2 \, \frac{\kappa \, e^{\Osc_{B(a,2r)}
V}}{n}\right] \, + \, 2 \, \frac{\kappa \, r^2 \, e^{\Osc_{B(a,2r)} V}}{n} \right) \, \int \,
|\nabla f|^2 \dmu \, .$$
\end{theorem}
\begin{proof}
The underlying stochastic process has for infinitesimal generator $L=\Delta - \nabla V. \nabla$.
We start with assuming that $V$ is smooth.

If $U=B(a,r)$, $U_r=B(a,2r)$ and we may find some function $\chi$ such that $\BBone_U \leq
\chi \leq \BBone_{U_r}$ with $|\nabla \chi|^2 \leq 2/r^2$. According to previous arguments and
the proof of theorem \ref{thmmaindiff} we know that $W(x)=\E_x[e^{\lambda \, T_U}]$ is a
Lyapunov function for $\lambda = \mu(U)/8C_P$. Hence for all smooth $f$
\begin{equation}\label{eqpoincdelocal}
\int \, f^2 \dmu \, \leq \, \frac 4\lambda \, \int \, |\nabla f|^2 \dmu + \left(\frac {8}
{r^2 \, \lambda} + 2\right) \, \int_{U_r} \, f^2 \dmu \, .
\end{equation}
It is well known that the Lebesgue measure satisfies a Poincar\'e inequality
\begin{equation}\label{poinclebesgue}
\int_{U_r} \, f^2 \, dx \, \leq \, \kappa \, \frac{r^2}{n} \, \int \, |\nabla f|^2 \, dx \, ,
\end{equation}
for all $f$ such that $\int_{U_r} f \, dx = 0$ for some universal constant $\kappa$ (we shall
revisit similar results later on). Accordingly, using a standard perturbation argument, we have
\begin{equation}\label{poinclocalball}
\int_{U_r} \, f^2 \dmu \, \leq \, \kappa \, \frac{r^2}{n} \, e^{\Osc_{U_r}V} \, \int_{U_r} \,
|\nabla f|^2 \dmu \, + \, \left(\int_{U_r} \, f \dmu\right)^2 \, .
\end{equation}
The result follows for smooth $V's$. In the general case we approximate $V$ and remark that, first
the Poincar\'e constants for the approximating measures converge to $C_P$ as well as the
Oscillation term with our assumptions.
\end{proof}

\begin{remark}\label{reminteret}
We state the result for balls $B(a,r)$ for simplicity. Of course the proof can be adapted to more
general subsets. Also notice that the possible interesting situations are for small $r's$.

If $A$ is a connected open domain the previous result applies to the uniform measure on $A$ and
$B(a,2r) \subset A$. The oscillation of $V$ is then equal to $0$.

\end{remark}

\bigskip

\section{Probability measures on the line.}\label{line}

In this section we shall look at the case $n=1$, $\mu(dx) = Z^{-1} \, e^{-V(x)} \, dx$ where
$Z$ is a normalization constant. Since the Poincar\'e constant is unchanged by translating the
measure we may also assume that $\int x d\mu =0$.

General bounds for the Poincar\'e constant are well known using Hardy-Muckenhoupt weighted
inequalities (see e.g. \cite{logsob}). Another approach was recently proposed in \cite{LLS}
where bounds for both the Poincar\'e constant and the exponential moment for hitting times are
obtained, through the rate function and speed measure. Notice that the results of section
\ref{secexit} seem to be less precise in the one-dimensional situation but cover all possible
dimensions.
\medskip

\subsection{Super-linear and log-concave one dimensional distributions.\\}\label{subsec1log}

Our interest here is to describe the Poincar\'e constant for particular $\mu$ including the
log-concave situation. The log-concave situation indeed deserved a lot of interest due to the
belief that, in the multidimensional isotropic case (namely the covariance matrix is the identity), it is close to the independent one. It is therefore particularly relevant to get bounds on functional inequalities in terms of the variance. For log-concave measures $\mu$ on the line
Bobkov (\cite{bob99}) proved that
\begin{equation}\label{eqbob1}
\Var_\mu(x) \leq C_P(\mu) \leq 12 \, \Var_\mu(x)
\end{equation}
where $x$ denotes the identity function. One can also look at another approach in \cite{Fou}.

In our previous work \cite{BBCG} we have shown how to use the Lyapunov function method to
recover the general result of Bobkov saying that any log-concave probability measure (in any
dimension) satisfies a Poincar\'e inequality. Here we shall be more precise for the one
dimensional case and we shall recover a bad version of Bobkov's result \eqref{eqbob1}, i.e.
with a worse constant larger than 12 but for more general measures. We start with some
definitions.

\begin{definition}\label{def1}
Let $\mu(dx) = Z^{-1} \, e^{-V(x)} \, dx$ be a probability measure on the line. We assume that
there exists $V_{min} > - \infty$ such that $V_{min} \leq V \leq +\infty$. Let $a\in \R$ such
that $a\in Argmin(V)$ (such an $a$ exists if for instance $V$ is continuous on $V<+\infty$).
We may assume that $V_{min}=\min V$.

For $\beta >0$ we denote by $R_+(\beta)$ any positive number such that $V(a+u) -V(a) \leq
\beta$ for all $0\leq u \leq R_+(\beta)$ and similarly $R_-(\beta)$ on the left hand side of
$a$. Finally $R(\beta)=R_+(\beta)\vee R_-(\beta)$.

We shall say that $V$ is $\beta$-superlinear if for $t\geq R_+(\beta)$ (resp. $t\geq
R_-(\beta)$) one has $$V(a+t) - V(a) \geq  \frac{c_\beta}{R(\beta)} \, t  - h_\beta \quad
\left(\textrm{resp. } V(a-t) - V(a) \geq \frac{c_\beta}{R(\beta)} \, t - h_\beta \right)$$ for
some non-negative constant $c_\beta$ and some $h_\beta$.
\end{definition}

\begin{remark}\label{remderiv}
Let $\mu(dx) = Z^{-1} \, e^{-V(x)} \, dx$ be a probability measure on the line, with $V$ of
class $C^1$. We assume that $\min V =0 = V(a)$ and that there exist $\beta >0$ and $\theta >0$
such that $\textrm{sign}(x-a) \, V'(x) \geq \theta$ outside some subset $N_\beta$ of the level
set $\{V \leq \beta\}$.

Since $\textrm{sign}(x-a) \, V'(x) \geq \theta$ outside $N_\beta$ it is easily seen that
$N_\beta$ is necessarily a closed interval. We thus choose $R_+(\beta)$ and $R_-(\beta)$ such
that $N_\beta=[a - R_-(\beta),a+R_+(\beta)]$. We may assume that $R_+(\beta)\geq R_-(\beta)$.
\smallskip

For $x \geq a + R_+(\beta)$, our assumptions furnish
\begin{eqnarray*}
V(x) & \geq & V(x) - V(a+ R_+(\beta)) \\ & \geq & \theta \, (x-a-R_+(\beta)) \\ & \geq &
 \frac{c_{\beta}}{R_+(\beta)} \, (x-a) - h_{\beta}
\end{eqnarray*}
where $c_{\beta}=\theta \, R_+(\beta)= h_{\beta}$.

For $x \leq a - R_-(\beta)$ we have the same result of course, still with $c_{\beta}=\theta \,
R_+(\beta)$ a priori with $h=\theta \, R_-(\beta)$ which is smaller than $h_\beta$, so that
the result still holds with $h_\beta$.

Hence $V$ is $\beta$-superlinear. Actually it is $\beta'$-superlinear for any $\beta'\geq
\beta$.

Our definition looks thus unnecessarily intricate. However, we shall see that is well
appropriate for the isotropic normalization.
\end{remark}

The next lemma allows us to compare the variance and the $\beta$ level set values,
\begin{lemma}\label{lem1}
Assume that $V$ is $\beta$-superlinear and that $\int x d\mu = 0$. , then
$$R_+^2(\beta) \vee R_-^2(\beta) \leq 12 \, \Var_\mu(x) \, e^{\beta} \, \left(1 + \, \frac{2 \, e^{h_\beta}}{c_\beta}\right) \,
,$$and $$ R_+^2(\beta) \vee R_-^2(\beta) \geq \frac 12 \, \Var_\mu(x) \, e^{-\beta} \,
\left(\frac 13 + \frac{e^{h_\beta}}{c_\beta} \, \left(1 + \frac{ 2}{c_\beta} +
\frac{2}{c_\beta^2}\right)\right)^{-1} \, .$$
\end{lemma}
The result is of course coherent with the previous remark \ref{remnorm}.
\begin{proof}
We may and will assume that $V(a)=0$ (just modify $Z$). We fix once for all $\beta$ and thus
skip the dependence in $\beta$ for notational convenience. Let $R=R_+ + R_-$ denote by
$\sigma^2$ the variance of $\mu$.

Since $V$ is $\beta$-superlinear we have
\begin{eqnarray*}
R(\beta) \, e^{-\beta} & \leq &  \int_{a-R_-}^{a+R_+} \, e^{-\beta} dx \leq
\int_{a-R^-}^{a+R_+} \, e^{-V(x)} dx \leq
 Z \\ & \leq & \int_{a-R_-}^{a+R_+} dx + e^{h} \, \left(\int_{a+R_+}^{+\infty} \, e^{- \frac{c}{R(\beta)} \, (x-a)} \, dx
  + \int_{-\infty}^{a-R_-} \, e^{- \frac{c}{R(\beta)} \, (a-x)} \, dx\right) \\ & \leq & R(\beta)
    \, \left(1 + 2 \, \frac{e^h}{c}\right) \, ,
\end{eqnarray*}
i.e. $$R(\beta) \, e^{-\beta} \leq Z \leq  R(\beta)  \, \left(1 + \frac{2 e^{h}}{c}\right) \,
.$$ By symmetry we may also assume that $R_+ \geq R_-$ so that it is enough to get an upper
bound for $R_+$. But
\begin{eqnarray*}
\frac{e^{-\beta}}{3} \, \left((a+R_+)^3 - a^3\right) & = & \int_{a}^{a+R_+} \, e^{-\beta} x^2
\, dx \leq \int_{a}^{a+R_+} \, x^2 \, e^{-V(x)} dx \leq
 Z \, \sigma^2 \, .
\end{eqnarray*}
Using $R(\beta)= R_+$ we thus obtain
\begin{equation}\label{eqvar1}
R_+^2 + 3a R_+ + 3a^2 \leq 3 \, \sigma^2 \, e^\beta \, \left(1 + \frac{2 e^{h}}{c}\right) \, .
\end{equation}
 If $a>0$ we
immediately obtain $R_+^2 \leq 3 \, \sigma^2 \, \left(1 + \frac{2 e^{h}}{c}\right)$. If $a\leq
0$ the minimal value of the left hand side in \eqref{eqvar1} (considered as a polynomial in
$a$) is obtained for $a = -R_+/2$ and is equal to $R_+^2/4$ so that we obtain in all cases
\begin{equation}\label{eqvar2}
R_+^2 \leq 12 \, \sigma^2 \, e^\beta \, \left(1 + \frac{2 e^{h}}{c}\right) \, .
\end{equation}
\smallskip

In the same way we see that if $a>0$ then $a^2 \leq 2 \, \sigma^2 \, e^\beta \, \left(1 +
\frac{2 e^{h}}{c}\right)$. If $a\leq 0$ the minimal value of the left hand side in
\eqref{eqvar1} (considered as a polynomial in $R_+$) is obtained for $R_+= - \frac 32 \, a$
and is equal to $3a^2/4$ so that we obtain in all cases
\begin{equation}\label{eqvar3}
a^2 \leq 4 \, \sigma^2 \, e^\beta \, \left(1 + \frac{2 e^{h}}{c}\right) \, .
\end{equation}
\smallskip

We turn to the second bound. Again we assume that $R_+\geq R_-$. Recall that $\sigma^2 \leq
\int (x-a)^2 d\mu$. Similarly to the first bound we can thus write
\begin{eqnarray*}
Z \, \sigma^2 & \leq &  \int (x-a)^2 \, e^{-V(x)} \, dx \\ & \leq & \int_{a-R_-}^{a+R_+} \,
(x-a)^2 \, dx + \, e^{h} \left(\int_{a+R_+}^{+\infty} \, (x-a)^2  e^{- \frac{c}{R(\beta)} \,
(x-a)} \, dx
  + \int_{-\infty}^{a-R_-} \, (x-a)^2  e^{- \frac{c}{R(\beta)} \, (a-x)} \, dx\right) \, ,\\ & \leq &
  \frac 13 \, (R_+^3+R_-^3) + e^h  \, \frac{R_+^3+ R_-^3}{c} \, \left(1 + \frac
2c + \frac{2}{c^2}\right) \, , \\ & \leq & 2 \, R_+^3 \, \left(\frac 13 + \frac{e^{h}}{c} \,
\left(1 + \frac 2c + \frac{2}{c^2}\right)\right) \, .
\end{eqnarray*}
Using $Z \geq R_+ e^{-\beta}$ we thus obtain
$$R_+^2 \geq \frac 12 \, \sigma^2 \, e^{-\beta} \, \left(\frac 13 + \frac{e^{h}}{c} \,
\left(1 + \frac 2c + \frac{2}{c^2}\right)\right)^{-1} \, .$$
\end{proof}

We turn to the study of the Poincar\'e constant.

\begin{remark}\label{remnorm} \quad If $m=\int x \, d\mu$ the measure $e^{-V(x+m)} dx$ is centered and
share the same Poincar\'e constant as $\mu$. Replacing $a$ by $a+m$ we may and will assume
that $m=0$.

Similarly if we consider the probability measure $\mu_u = u \, e^{-V(ux)} dx$, we have $u^2 \,
\Var_{\mu_u}(x)= \Var_\mu(x)$ and an easy change of variables shows that $u^2 \, C_P(\mu_u) =
C_P(\mu)$. So we can assume that $\Var_\mu(x)=1$.

If $V$ is $\beta$-superlinear, it is easy to see that $V_u$ (defined by $V_u(x)=V(ux)$) is
still $\beta$-superlinear, with the same constants $c_\beta$ and $h_\beta$, but replacing
$R_\beta$ by $R_u(\beta)=R_\beta/u$. \hfill $\diamondsuit$
\end{remark}
\medskip

\begin{proposition}\label{proppoincbeta}
Let $\mu(dx) = Z^{-1} \, e^{-V(x)} \, dx$ be a probability measure on the line, with $V$ of
class $C^1$. We assume that $\min V =0 = V(a)$ and that there exist $\beta_0 >0$ and $\theta
>0$ such that $\textrm{sign}(x-a) \, V'(x) \geq \theta$ outside some subset $N_{\beta_0}$ of
the level set $\{V \leq \beta_0\}$.

Then there exists a constant $C(\beta_0,\theta)$ such that the Poincar\'e constant $C_P(\mu)$
satisfies $$C_P(\mu) \leq C(\beta_0,\theta) \, \Var_\mu(x) \, .$$
\end{proposition}
\begin{proof}
As we already remarked we can  assume that $\mu$ is centered, and $V$ being of $C^1$ class, we
have $Lg=g'' - V'g'$.

We know that $V$ is $\beta$-superlinear for any $\beta \geq \beta_0$. We denote by
$N_\beta=[a-R_-(\beta),a+R_+(\beta)]$. We shall modify $\mu$ introducing
\begin{equation}\label{eqmumodif}
\mu_\beta(dx) = Z_\beta^{-1} \, \left(e^{-V(x)} \, \BBone_{x\notin N_\beta} + e^{-\beta} \,
\BBone_{x \in N_\beta}\right) \, dx.
\end{equation}
Note that, according to  Lemma \ref{lem1} $$1 \leq \frac{Z}{Z_\beta} \leq e^{\beta} \left(1 +
\frac{2 \, e^{h_\beta}}{c_\beta}\right) \, .$$ It follows that
\begin{equation}\label{eqcomparemodif}
e^{-\beta} \leq \frac{d\mu_\beta}{d\mu} \leq  e^{2\beta} \left(1 + \frac{2 \,
e^{h_\beta}}{c_\beta}\right) \, .
\end{equation}
Accordingly we know that
\begin{equation}\label{eqpoinccompare}
C_P(\mu) \leq e^{3\beta} \left(1 + \frac{2 \, e^{h_\beta }}{c_\beta}\right) \, C_P(\mu_\beta)
\, .
\end{equation}
It remains to find an estimate for the Poincar\'e constant of $\mu_\beta$.
\smallskip

We have to face a small problem since the potential $V_\beta$  of $\mu_\beta$ is no more of
class $C^1$, but we still have a drift, i.e. $V'(x) \BBone_{x\notin N_\beta}$, and an easy
approximation procedure allows us to extend Theorem \ref{thmmaindiff} in this situation.

We denote by $L_\beta$ the associated generator i.e. $L_\beta f (x) = f''(x) - V'(x)
\BBone_{x\notin N_\beta} f'(x)$.

We shall now introduce a well chosen Lyapunov function. We define $$u(x) = |x| \,
\BBone_{|x|>1} + \left(\frac 38 + \frac 34 \, x^2 - \frac 18 \, x^4\right) \, \BBone_{|x|\leq
1} \, .$$ It is easily seen that $u$ is of class $C^2$.

Now for $a_\beta = a + \frac{R_+(\beta) - R_-(\beta)}{2}$ (which is the center of $L_\beta$),
and $R=R_+(\beta) + R_-(\beta)$ we define $$W_\beta(x) = e^{\gamma \, u(x-a_\beta)} \, .$$ An
easy calculation shows that
$$L_\beta W_\beta \leq \gamma (\gamma - \theta) W_\beta \textrm{ if } |x-a| \geq R \, .$$
Choosing $\gamma = \theta/2$, it follows that $W_\beta$ is a Lyapunov function i.e. satisfies
(H1) with $$\lambda = \frac 12 \, \theta^2 = \frac 12 \, \frac{c_{\beta}^2}{R^2} =
\frac{v(\beta)}{\Var_\mu(x)} \, ,$$ according to Lemma \ref{lem1}.

It is thus enough to apply \eqref{eqpoinclyap} with some care. First we replace $U$ by
$N_\beta$, then $U_r$ by $N_{2\beta}$. We may thus choose some $\chi$ such that
$\Gamma(\chi,\chi)$ is of order $(R_{2\beta}-R_\beta)^{-2}$ i.e. such that
$\Gamma(\chi,\chi)/\lambda$ only depends on $\beta$ (and not explicitly on $\Var_\mu(x)$).

Since $\mu_\beta$ is uniform on $N_\beta$, it is known that its Poincar\'e constant (in
restriction to $N_\beta$) is equal to $R^2/\pi^2$, and again thanks to Lemma \ref{lem1} it is
bounded independently of $V$ by some constant that only depends on $\beta$ and $\lambda$.

The proof is completed, and the reader easily sees why we did not give an explicit value for
the constant $C(\beta,\lambda,\Var_\mu(x))$.

\end{proof}
\medskip

\begin{remark}
The previous proposition is not surprising. It tells us that once the exponential
concentration (which is a consequence of the Poincar\'e inequality) rate at infinity is known,
and the bound of the density is given (at finite horizon), the Poincar\'e constant has to be
controlled up to the natural scaling in the variance. We have given a proof of this natural
conjecture under a strong form of concentration result. This result entails in particular
double-well potentials. Note that no bound on the second derivative is needed (except that the
first derivative has to stay greater than $\lambda$), so that the previous result contains
much more general situations than the log-concave situation. We may even build examples with a
Bakry-Emery curvature equal to $- \infty$.
\end{remark}
\smallskip

We turn to the log-concave situation. Since our method covers more general situations, it is
certainly not sharp. So it is an illusion to hope to recover the constant $12$ in Bobkov's
result. Hence we shall even not try to give an explicit constant.

\begin{theorem}\label{thm1d}
There exists a universal constant $C$ such that for all log-concave probability measure $\mu$
on the real line, $$C_P(\mu) \leq C \, \Var_\mu(x) \, .$$
\end{theorem}
\begin{proof}
According to Remark \ref{remnorm} the result will follow if we prove the existence of the
universal constant $C$ for any log-concave measure with $\Var_\mu(x)=1$.

First we assume that $\mu(dx) = Z^{-1} \, e^{-V(x)} \, dx$ is a log-concave probability
measure on the line, with $V$ a $C^1$ function. We assume that $\min V =0 = V(a)$ and
$\Var_\mu(x)=1$.

Since $V$ is convex it is easily seen that for any $\beta>0$, $N_\beta$ is necessarily a
closed interval denoted again $[a-R_-(\beta),a+R_+(\beta)]$.

In particular if $x \geq a + R_+(2\beta)$, the convexity of $V$ yields
\begin{eqnarray*}
V(x) & \geq & V(x) - V(a+ R_+(\beta)) \\ & \geq & \frac{\beta}{R_+(2\beta)-R_+(\beta)} \,
(x-a-R_+(\beta)) \\ & \geq & \frac{\beta}{R_+(2\beta)} \, (x-a-R_+(\beta)) = \frac{c^+_{2
\beta}}{R_+(2\beta)} \, (x-a) - h_{2\beta}
\end{eqnarray*}
where $c^+_{2\beta}=\beta$ and $0\leq h_{2 \beta} \leq \beta$.

For $x \leq a - R_-(2\beta)$ we have a similar result replacing $R_+$ by $R_-$ hence with
$c_-(2 \beta)=\beta$ again. Since $R_+$ and $R_-$ are both smaller than (or equal to)
$R(\beta)$, $V$ is $2\beta$-superlinear and Lemma \ref{lem1} yields $$R^2(2\beta) \leq 12 \,
e^{\beta} \, \left(1 + \frac {2 e^{2\beta}}{\beta}\right) \, .$$ In addition for $x \geq a +
R_+(2\beta)$ convexity yields
$$V'(x) \geq \frac{\beta}{R_+(2\beta)} \geq \frac{\beta}{R(2\beta)} \geq \frac{\beta^{3/2} \, e^{-\beta/2}}{2 \sqrt 3 \, (\beta
+2 e^\beta)^{1/2}} \, ,$$ and the same result is true for $x \leq a - R_-(2\beta)$.
Proposition \ref{proppoincbeta} yields a bound for each $\beta$ ($\beta=1$ for example). We
should optimize in $\beta$ but as we said we shall never attain the optimal bound 12.
\smallskip

In the general case ($V$ convex with values in $]-\infty,+\infty]$) we first approximate $V$
by everywhere finite convex functions, and then approximate such a function by a smooth one
convoluing it for instance with gaussian kernels with small variance.
\end{proof}
\medskip

\subsection{Hardy type inequalities.\\}\label{subsecdelocal}

In the spirit of Remark \ref{remstokes} we can state another particular result of Hardy type,
which is known to hold (with a better constant) if $b$ below is a median of $\mu$
\begin{theorem}\label{thmavecunpoint}
Let $d\mu= e^{-V(x)} \, dx$ be a probability measure on the real line satisfying a Poincar\'e
inequality with constant $C_P$. We assume that there exist a sequence $V_n$ of $C^1$ functions
such that $e^{-V_n}$ converges to $e^{-V}$ weakly in $\sigma(\L^1,\L^\infty)$. Then for all
$b\in \R$ the following inequality holds for all bounded smooth $f$ ,
\begin{equation}\label{eqpoincsob}
\int \, (f(x)-f(b))^2 \, \mu(dx) \, \leq \, \frac{8 C_P}{\mu(]-\infty,b]) \wedge
\mu([b,+\infty[)} \, \int \, (f')^2(x) \, \mu(dx) \, .
\end{equation}
\end{theorem}
\begin{proof}
Assume first that $V$ is of class $C^1$. If $X_t$ is the diffusion process with generator $Lf
= f'' - V' \, f'$ (which is conservative since Poincar\'e inequality, hence (H1) holds),
Proposition \ref{propentry} tells us that for any $\theta < \mu(]-\infty,b]/ 8C_P$ the hitting
time of $]-\infty,b]$ has an exponential moment of order $\theta$. Hence one can find a
Lyapunov function satisfying $LW = - \theta \, W$ on $[b,+\infty[$, namely $W(x)= \mathbb
E_x(e^{\theta \, T_b})$. It follows that for a smooth $f$ and $A\geq b$,
\begin{eqnarray*}
\int_b^{A} \, (f(x)-f(b))^2 \, \mu(dx) & = & \frac{- 1}{\theta} \, \int_b^{A} \,
\frac{LW}{W}(x) \, (f(x)-f(b))^2 \, \mu(dx) \\ & \leq & \frac{1}{\theta} \, \left(\int_b^{A}
\, (f'(x))^2 \, \mu(dx) -  \left((f(A)-f(b))^2 \, \frac{W'(A)}{W(A)} \,
e^{-V(A)}\right)\right)
\end{eqnarray*}
the latter being obtained as in Remark \ref{remstokes} using integration by parts, since
$f(x)-f(b)=0$ for $x=b$. But $W$ is clearly non-decreasing in $x$ so that the last term into
braces is non-negative, yielding the bound we claimed on $[b,+\infty[$ by letting $A$ go to
$+\infty$. The same holds on the left hand side of $b$.

Hence the Hardy-Poincar\'e-Sobolev inequality \eqref{eqpoincsob} holds for any constant larger
than $$\frac{8 C_P}{\mu(]-\infty,b]) \wedge \mu([b,+\infty[)}$$ hence with this value by
taking the infimum, and then for a non-necessarily smooth $V$ using an approximation
procedure.
\end{proof}
\medskip

As it is clear in the previous proof we may replace the full $\R$ by any interval containing
$b$ without any change in the constant. Since the Variance of $f$ in restriction to an
interval minimizes the square distance to any value, we thus obtain as a corollary

\begin{corollary}\label{corenunpoint}
Let $d\mu= e^{-V(x)} \, dx$ a probability measure on the real line satisfying a Poincar\'e
inequality with constant $C_P$. Then for all interval $(a,b) \subseteq \R$ the following
inequality holds for $\mu_{(a,b)}$ the restriction of $\mu$ to $(a,b)$ and for all bounded
smooth $f$ ,
\begin{equation}\label{eqpoincsobbis}
\Var_{\mu_{(a,b)}}(f) \, \leq \, \frac{8 C_P}{\sup_{u \in (a,b)} \{\mu(]-\infty,u]) \wedge
\mu([u,+\infty[)\} }\, \int \, (f')^2(x) \, \mu_{(a,b)}(dx) \, .
\end{equation}
In particular if $a \leq m_\mu \leq b$ then $\mu_{(a,b)}$ satisfies a Poincar\'e inequality
with a constant not bigger than $16 \, C_P$.
\end{corollary}

This bound can certainly be attained and improved by looking carefully at Muckenhoupt type
constants, at least when the median belongs to the interval.
\medskip

\subsection{$\L^1$ inequalities.\\}\label{subsecl11}

It is well known that one obtains a stronger inequality replacing the $\L^2$ norm by a $\L^p$
norm for $1\leq p \leq 2$ (see e.g. \cite{bobkov-zeg} chapter 2). The $\L^1$ Poincar\'e
inequality (sometimes called Cheeger inequality) is of particular interest since it yields
controls for the isoperimetric constant (see e.g. \cite{BH97,bob99}). Due to the standard
\begin{equation}\label{eqmedvar1}
\frac 12 \, \int |f -\mu(f)| \, d\mu \leq \int |f -m_\mu(f)| \, d\mu \leq \int |f -\mu(f)| \,
d\mu \, ,
\end{equation}
where $\mu(f)$ and $m_\mu(f)$ denote respectively the mean and a $\mu$ median of $f$, such an
inequality can be written indifferently
\begin{equation}\label{eqcheeg}
\int |f -\mu(f)| \, d\mu \leq C_C \, \int \, |\nabla f| \, d\mu \quad \textrm{ or } \quad \int
|f -m_\mu(f)| \, d\mu \leq C'_C \, \int \, |\nabla f| \, d\mu \, .
\end{equation}

\eqref{eqcheeg} is true for any log-concave distribution and actually $C_C$ and $C_P$ differ
by an universal multiplicative constant (see \cite{ledgap}). For one dimensional log-concave
distribution $C_C$ is universally bounded (see \cite{BH97}). In our previous paper
(\cite{BBCG}) we have shown that the existence of a Lyapunov function $W$ as in (H1) implies a
Cheeger type inequality, provided $\nabla W/W$ is bounded.
\smallskip

We shall here derive such an inequality, with the correct normalization factor $\mu(|x
-\mu(x)|)$ which immediately follows by a linear change of variables in \eqref{eqcheeg}.

\begin{theorem}\label{thmcheeger1}
Under the hypotheses of Proposition \ref{proppoincbeta} there exists a constant
$C(\beta,\lambda,\mu(|x -\mu(x)|))$  such that the Cheeger constant $C_C(\mu)$ satisfies
$$C_C(\mu) \leq C(\beta,\lambda,\mu(|x -\mu(x)|))  \, .$$ In particular if $\mu$ is a log-concave
probability measure on the line, there exists an universal constant $C$ such that $C_C(\mu)
\leq C \, \mu(|x-\mu(x)|)$.
\end{theorem}

\begin{proof}
We follow the proof of Proposition \ref{proppoincbeta} (see the notations therein) proving a
Cheeger inequality for the measure $\mu_\beta$. Recall that $W_\beta$ satisfies $L_\beta
W_\beta \leq -\lambda^2/4 \, W_\beta + b(R,\beta,\lambda) \BBone_{N_\beta}$.
\smallskip

The first thing to do is to show that $R_+\vee R_-$ is controlled, from above and from below
by a quantity depending only on $\beta$, $\lambda$ and $\mu(|x-\mu(x)|)$ i.e. to prove the
analogue of Lemma \ref{lem1}. Denote by $s$ the quantity $\mu(|x-\mu(x)|)$. Then mimiking the
proof of Lemma \ref{lem1} we can prove $$R_+\vee R_- \leq 2 (1 + \frac{2e^h}{c}) \, s \, ,$$
and $$R_+\vee R_- \geq \frac 12 \, s \, e^{-\beta} \, C(h,c) \, ,$$ for some $C(h,c)>0$.
\smallskip

Now we may assume that $s=1$. The second thing to do is to recall the reasoning in \cite{BBCG}
i.e. if $f$ is smooth and $g=f-m$ for some constant $m$ we have
\begin{eqnarray*}
\int |g| \, \mu_\beta(dx) & \leq & \frac{4}{\lambda^2} \, \int |g| \left(- \, \frac{L_\beta
W_\beta}{W_\beta}\right) \, \mu_\beta(dx) + \frac{4 b(\beta,\lambda) e^\beta}{\lambda^2 \,
Z_\beta} \, \int_{N_\beta} \, |g| \, dx \\ & \leq & \frac{4}{\lambda^2} \, \int \left(|g'|
\left( \frac{|W'_\beta|}{W_\beta}\right) - |g|
\left(\frac{|W'_\beta|^2}{W^2_\beta}\right)\right) \, \mu_\beta(dx) + \frac{4 b(\beta,\lambda)
e^\beta}{\lambda^2 \, Z_\beta} \, \frac{R}{\pi} \, \int_{N_\beta} \, |g'| \, dx
\end{eqnarray*}
if we choose $m=\int_{N_\beta} \, f(x) \, dx$. The first term is obtained after integrating by
parts, the second one is using the standard Cheeger inequality for the uniform measure on an
interval.

Now remark that $|W'_\beta|/W_\beta$ is bounded by some constant depending only on $\beta$ and
$\lambda$. Finally we have obtained (if $\mu(|x-\mu(x)|)=1$), $$\int |f-\mu_\beta(f)|
d\mu_\beta \leq 2 \, \int |f-m_{\mu_\beta}(f)| d\mu_\beta \leq 2 \, \int |g| \, d\mu_\beta
\leq K(\beta,\lambda) \, \int \, |f'| \, d\mu_\beta  \, ,$$ hence the result for $\mu_\beta$
and then for $\mu$ as in Proposition \ref{proppoincbeta}.

The log-concave case is then similar to Theorem \ref{thm1d}.
\end{proof}
\medskip

As we already said the previous Theorem contains Proposition \ref{proppoincbeta} thanks to
Cheeger's inequality $C_P \leq \, 4 \, C_C^2$. Actually our proof yields so bad constants in
both cases that it is really tedious to check when the previous relation gives a better bound
than Proposition \ref{proppoincbeta}. We also insist on the proof of both properties using
Lyapunov function. As we have seen, the proof of a $\L^1$ inequality requires the boundedness
of $W'/W$. In particular if we choose for $W$ the Laplace transform of hitting times
$\E_x(e^{\theta \, T_b})$, this latter property is not ensured. So we cannot obtain similar
results as in subsection \ref{subsecdelocal}.

\bigskip

\section{$\phi$ moments and Poincar\'e like inequalities.}\label{secweak}

Since the status of the existence of exponential moments for hitting times is now characterized
through the results of section \ref{secexit}, it is certainly interesting to look at more general
$\phi$ moments. The first result is a direct consequence of \eqref{eqipp}:

\begin{proposition}\label{propexpbest}
Assume that $L$ is uniformly strongly hypoelliptic. If $U$ is an open connected set with
$\mu(U)<1$, then $$\sup \left\{\lambda  \, , \textrm{ such that } \, \E_\mu\left(e^{\lambda
T_U}\right)<+\infty\right\} \, < \, +\infty \, .$$ In particular if $\phi$ growths faster, at
infinity, that any exponential $\E_\mu\left(\phi(T_U)\right) = +\infty$ .
\end{proposition}
\begin{proof}
We already saw that, in the uniform strong hypoelliptic situation, $\E_x\left(e^{\lambda
T_U}\right) < +\infty$ for all $x$ as soon as $\E_\mu\left(e^{2 \lambda T_U}\right)<+\infty$.
According to the proof of Theorem \ref{thmmaindiff}, we thus know that there exists a Lyapunov
function satisfying (H1). \eqref{eqipp} implies that $$\int f^2 \dmu \leq \frac{1}{\lambda} \,
\int \, \Gamma(f,f) \dmu$$ for all smooth $f$ with support in $\bar{U}^c$. This cannot hold for
all $\lambda$ since $\mu(U)<1$, just looking at $\lambda \to +\infty$.
\end{proof}

This result is in accordance with the fact that one cannot improve on the exponential convergence
in $\L^2$ (or in total variation distance) even for very strong repelling forces. In order to
discriminate diffusions satisfying a Poincar\'e inequality, one has to introduce new inequalities
(e.g. $F$-Sobolev inequalities or super-Poincar\'e inequalities) or contraction properties of the
semi-group (see e.g. \cite{BCR1,CGWW}). Another connected possibility is to look at exponential
decay to equilibrium for weaker norms than $\L^p$ norms (see e.g. \cite{CatGui3}). The certainly
best known case is the one when a logarithmic Sobolev inequality holds or equivalently the
semi-group is hypercontractive or equivalently exponential convergence holds in entropy (or in $\L
\log \L$ Orlicz norm) (see e.g. \cite{logsob} for an elementary introduction to all these
notions).

The use of Lyapunov functions for studying such stronger inequalities is detailed in \cite{CGWW}.
It should be very interesting to understand these phenomena in terms of hitting times. We strongly
suspect that what is important is the behavior of $W(x)= \E_x\left(e^{\lambda T_U}\right)$ as $x$
goes to infinity. For instance if $W$ is bounded, we suspect that the semi-group is
ultracontractive (or more properly ultrabounded). Some hints in this direction are contained in
\cite{Cat07} Theorem 7.3 at least for diffusion processes on the real line. Let us state a result
in this direction:

\begin{proposition}\label{ultra}
Assume that $L=\Delta - \nabla V.\nabla$, where $V$ is smooth, is defined on $\R$.
$\mu(dx)=e^{-V(x)} dx$ (supposed to be a probability measure) is then symmetric for $L$. Assume in
addition that $|\nabla V|^2 - \Delta V \geq -C$ for some non-negative constant $C$.

Then there is an equivalence between
\begin{enumerate}
\item  the associated semi-group $P_t$ is ultrabounded (i.e. $P_t$ maps continuously
$\L^1(\mu)$ in $\L^\infty(\mu)$ for all $t>0$), and there exists an open interval $U$ such
that for all $x\in \mathbb R$, $P_x(T_U<+\infty)=1$, \item  there exists a bounded Lyapunov
function $W$, \item there exists an open interval $U$ and $\lambda>0$ such that
$$\sup_x \, \E_x\left(e^{\lambda \, T_U}\right) \, < \, +\infty \, .$$
\end{enumerate}
\end{proposition}
\begin{proof}
The equivalence between (2) and (3) follows from the proof of Theorem \ref{thmmaindiff}, since $L$
is uniformly elliptic.

If (1) holds, it follows from the arguments in Appendix B of \cite{CM}, that there exists an
unique quasi limiting distribution for the process (starting from the right of $U$) killed
when hitting any interior point of $U$. For all definitions connected with quasi-stationary
measures and quasi-limiting distributions we refer to \cite{Cat07,CM}. The same holds for the
process coming from the left of $U$. According to \cite{Cat07} Theorem 7.3, this implies that
the killed process ``comes down from infinity'' i.e. satisfies (3).

Conversely, \cite{Cat07} Theorem 7.3 tells us that (3) implies the condition (called (H5) therein)
$$\int_a^{+\infty} \, e^{-V(y)} \, \int_a^y \, e^{V(z)} \, dz \, dy \, < \, +\infty \, ,$$ for
$a=\sup U$. Define, for $x>a$,  $u(x)=\mu([x,+\infty[)$ and $$F(z) = z \,
\left(\int_a^{u^{-1}(1/z)} \, e^{V(y)} \, dy\right)^{-1} \, .$$ $z \mapsto F(z)/z$ is thus
non-increasing and we have
\begin{equation}\label{eqhardy}
u(x) \, F\left(\frac{1}{u(x)}\right) \, \int_a^x \, e^{V(y)} \, dy \, = \, 1 \, .
\end{equation}
According to results in \cite{BCR1} (see Remark 3.3 in \cite{CatGui3}), $\mu$ satisfies a
$\tilde{F}$-Sobolev inequality for a slight modification of $F$. Condition (H5) of \cite{Cat07}
recalled above implies that
$$\int^{+\infty} \frac{1}{u \, F(u)} \, du \, < +\infty \, .$$ The same holds with $\tilde{F}$ in
place of $F$. According to a result of \cite{Wbook} explained p.135 of  \cite{CatGui3}, this
implies that the semi-group is ultrabounded.
\end{proof}
\medskip

\subsection{Weak Poincar\'e inequalities and polynomials moments.}\label{subsecweak}

In this section we shall look at the existence of $\phi$-moments for functions $\phi$ growing
slower at infinity than an exponential, and actually we shall mainly focus on power functions. In
all the discussion below we shall assume, for simplicity, that $L$ is uniformly strongly
hypoelliptic and our symmetry assumption.
\smallskip

First of all recall that under our assumptions, defining for $q\in \N$,
\begin{equation}\label{eqmomq}
v_q(x) = \E_x \left(T_U^q\right)
\end{equation}
and provided $v_q$ is well defined for all $x$, $v_q$ is smooth and satisfies for $q\geq 1$
\begin{equation}\label{eqalmostlyap}
Lv_q(x) = - q \, v_{q-1}(x) \quad \textrm{ for all } x\in U^c \, ,
\end{equation}
as a simple application of the Markov property. We thus have some ``nested'' Lyapunov functions.

Henceforth we assume that $U$ is bounded (which is clearly not a restriction). Then, since $v_q(x)
>0$ when $d(x,U)\geq 1$, the Markov property together with the continuity of $v_q$ and the
compactness of $\{d(x,U)=1\}$, show that there exists $\kappa>0$ such that for all $x$ with
$d(x,U)\geq 1$, $v_q(x) \geq \kappa$ and $v_{q-1}(x)\geq\kappa$. Remark that equality
(\ref{eqalmostlyap}) is still true  for all $x$ such that $d(x,U)\ge 1$. Note that $v_{q-1}\le
v_q$ and we set $v_{q-1}(x)/v_q(x)=1$ if $v_q(x)=0$. We then obtain the following consequences

\begin{proposition}\label{propzarbi}
\begin{itemize}
\item[(1)] Weak Poincar\'e like inequalities.\\
Assume that a local Poincar\'e inequality holds. Suppose that $v_q(x)$ is finite for all $x$.
Then for all positive $s<1$,  there exists a positive function $\beta$ such that for all
bounded $f$
\begin{equation}
\Var_\mu(f)\le \beta(s)\int\Gamma(f,f)d\mu+s \Osc(f)^2
\end{equation}
and $\beta(s)=C\,\left(\inf\{u\,;\,\mu(\frac{v_{q-1}}{v_q}<u)>s\}\right)^{-1}$ for some explicit constant $C$.
\item[(2)] \quad Assume that $L=\Delta - \nabla V.\nabla$, where $V$ is smooth, is defined on
$\R$. Then, if $v_1$ is bounded, the associated semi-group $P_t$ is ultrabounded (hence for some
$\lambda>0$, $\sup_x \E_x(e^{\lambda \, T_U} < +\infty$).
 \item[(3)] \quad If there exists $C$
such that $v_q(x) \leq C \, v_{q-1}(x)$ for all $x$ with $d(x,U)\geq 1$, then $\mu$ satisfies a
Poincar\'e inequality (and consequently $T_U$ has some exponential moment for all $\P_x$).
\end{itemize}
\end{proposition}
The first part of the theorem gives that in the reversible setting, finiteness of moments of return times implies a
weak Poincar\'e like inequality, a result that we are not aware of in discrete times. It is however very difficult to
 get precise estimates of $\beta$ as we need concentration properties of $\mu$ and sharp control of $v_q$ and $v_{q-1}$.
 Using that $v_{q-1}\le v_q^{q-1\over q}$ we may get a lower bound for $\beta$ using only $v_q$.\\
Note that the second part of the Proposition is not so surprising and corresponds to the similar
discrete situation of birth and death processes on the half line (see Proposition 7.10 in
\cite{Cat07}). The third part is only expressing that $v_q$ is a Lyapunov function.
\begin{proof}
The first part of the theorem, inspired by \cite{CGGR} and the proof of the main theorem, may be proved in two steps that
 we sketch here. First, using the Lyapunov conditions (\ref{eqalmostlyap}) and the same line of reasoning than (H1)
 implies (H4) in our main theorem, we get some weighted Poincar\'e inequality: for some constant $C$, we have
$$\inf_a\int \frac{v_{q-1}}{v_q}(f-a)^2d\mu\le C\,\int\Gamma(f,f).
$$
Then, with $a_f=\mu(f\frac{v_{q-1}}{v_q})/\mu(\frac{v_{q-1}}{v_q})$, for all bounded $f$ and for all $u>0$
\begin{eqnarray*}
\Var_\mu(f)&\le&\int (f-a_f)^2d\mu\\
&=&\int_{\frac{v_{q-1}}{v_q}\ge u}(f-a_f)^2d\mu+\int_{\frac{v_{q-1}}{v_q}< u}(f-a_f)^2d\mu\\
&\le& u^{-1} \inf_a\int \frac{v_{q-1}}{v_q}(f-a)^2d\mu +\mu\frac{v_{q-1}}{v_q}<u)\Osc(f)^2
\end{eqnarray*}
which gives the result.

For the second part just remark that $$Lv_1(x) \leq - \, \frac{v_1(x)}{\sup v_1} \quad \textrm{ for
} x \in \bar{U}^c \, .$$ Hence $v_1/\kappa$ is a bounded Lyapunov function satisfying (H1) (with
$\bar{U}^c$ replaced by $\{d(x,U)>1\}$) and we may apply Proposition \ref{ultra}.

For the third part we similarly have $$Lv_q(x) \leq \, - \, \frac{v_{q-1}(x)}{v_q(x)} \, v_q(x)
\quad \textrm{ for } x \textrm{ such that } d(x,U)\geq 1 \, .$$ Hence $v_q/\kappa$ is a Lyapunov
function satisfying (H1) (with $\bar{U}^c$ replaced by $\{d(x,U)>1\}$), and we may apply Theorem
\ref{thmmaindiff}.
\end{proof}
\medskip

An immediate generalization of (3) in the previous proposition is the following assumption : there
exists an increasing function $\varphi$ growing to infinity and $R>0$, such that
\begin{equation}\label{eqlyapfaible}
\varphi(v_q(x))  \, \leq  \, q \, v_{q-1}(x) \quad \textrm{ for } |x|\geq R \, .
\end{equation}
Indeed if \eqref{eqlyapfaible} holds, we have $$Lv_q(x) \leq -  \, \varphi(v_q(x))$$ for $|x|$
large enough, and $v_q$ is thus a $\varphi$-Lyapunov function in the terminology of \cite{DFG} and
\cite{CGGR} (see definition 2.2 in the latter reference).
\medskip

Conversely, is it possible to get the existence of $\phi$-moments starting from a functional
inequality ? The first answer to this question was given in \cite{Math} where some Nash type
inequalities are shown to imply the existence of moments. The proof uses the fact that the Laplace
transform of $T_U$, $h_t^U(x)=\E_x(e^{- t \, T_U})$ satisfies $Lh - th =0$ for all $t>0$.
\medskip

Using the results in section 3 of \cite{CatGui2} again we can derive similar (actually stronger)
results. Recall that
$$
\P_\mu(T_U>t)  \leq  \P_\mu \left(- \, \frac 1t \, \int_0^t \, \BBone_U(X_s) \, ds + \mu(U) \geq
\mu(U)\right) \, .
$$
According to Proposition 3.5 in \cite{CatGui2} we thus have for $t$ large enough, $$\P_\mu(T_U>t)
\leq C(k) \, t^{-k} \, (\mu(U))^{-2k}$$ provided the process is $\alpha$-mixing with a mixing rate
$\alpha(u) \leq C (1+u)^{-k}$ for some integer $k\geq 1$.
\medskip

The mixing rate is connected to the rate of convergence to equilibrium of the semi group, as
explained in \cite{CatGui2}. This rate of convergence can be bounded using either a Weak
Poincar\'e Inequality (see \cite{rw}) or a $\varphi$-Lyapunov function (see \cite{DFG,BCG}). Let
us collect these results in the next (and final) theorem

\begin{theorem}\label{thmpmoment}
Assume that $L$ is uniformly strongly hypoelliptic. Let $U$ be a bounded connected open set.
Assume in addition one of the following conditions,
\begin{enumerate}
\item \quad $\mu$ satisfies a weak Poincar\'e inequality, i.e. there exists a non-increasing
function $\beta$ such that for all $s>$ and all bounded and smooth $f$, $$\Var_\mu(f) \leq
\beta(s) \, \int \, \Gamma(f,f) \dmu + s \, \Osc^2(f)$$ where $\Osc(f)$ denotes the oscillation
of $f$; in which case the process is $\alpha$-mixing with a mixing rate $$\alpha(t) \leq
\left(\inf \{s>0 \, ; \, \beta(s) \, \log(1/s) \, \leq \, t/2\}\right)^2 \, .$$ \item \quad there
exists a $\varphi$-Lyapunov function $W$  for some smooth increasing concave function $\varphi$
with $\varphi' \to 0$ at infinity; in which case the process is $\alpha$-mixing with a mixing rate
$$\alpha(t) \leq C \, \left(\int W \dmu\right) \, \frac{1}{\varphi \circ H^{-1}_\varphi (t)} \, \,
 ,$$ where $H_\varphi(t)=\int_1^t (1/\varphi(s)) ds$ and we assume that $\int W \dmu < +\infty$.
\end{enumerate}
If in addition $\alpha(t) \leq C \, (1+t)^{- k}$ for some positive integer $k$, then
$$\P_\mu(T_U>t) \leq C(k) \, t^{-k} \, (\mu(U))^{-2k}$$ for some constant $C(k)$.

In particular for all $j<k$, $\E_\mu(T_U^j) < +\infty$. The same holds for $\mu$ almost all $x$,
and for all $x$, and $j<k/2$, $\E_x(T_U^j) < +\infty$.
\end{theorem}

The interested reader will find in \cite{BCG,CGGR} in particular many examples (including the so
called $\kappa$-concave measures) of measures satisfying one (or both) of the previous conditions.

\bigskip

\bibliographystyle{plain}

\def\cprime{$'$}

\end{document}